\documentclass[12pt]{article}
\usepackage{amssymb}
\usepackage{amsmath,amsthm}
\usepackage[latin1]{inputenc}
\usepackage{hyperref}
\usepackage{color}
\usepackage{graphicx}
\hypersetup{colorlinks=true, linkcolor=blue, citecolor=blue, urlcolor=blue}

 \newtheorem{remark}{Remark}

 \newtheorem{theorem}[remark]{Theorem}
 \newtheorem{proposition}[remark]{Proposition}
 \newtheorem{corollary}[remark]{Corollary}

\title{Alliance free sets in Cartesian product graphs}

\author{ Ismael G. Yero$^{1}$, Juan A. Rodr\'{\i}guez-Vel\'{a}zquez$^{1}$,
\\ Sergio Bermudo$^{2}$
\\
$^1${\small Departament d'Enginyeria Inform\`atica i Matem\`atiques,}\\
{\small Universitat Rovira i Virgili,}  {\small Av. Pa\"{\i}sos
Catalans 26, 43007 Tarragona, Spain.} \\{\small
ismael.gonzalez\@@urv.cat, juanalberto.rodriguez\@@urv.cat}
\\$^2${\small Department of Economy, Quantitative Methods and Economic
History} \\ {\small Pablo de Olavide University, Carretera de Utrera
Km. 1, 41013-Sevilla, Spain }\\ {\small sbernav\@@upo.es}
}

\begin{document}

\maketitle

\begin{abstract}

Let $G=(V,E)$ be a graph. For a non-empty subset of vertices $S\subseteq V$, and vertex $v\in V$, let $\delta_S(v)=|\{u\in S:uv\in E\}|$ denote the cardinality of the set of neighbors of $v$ in $S$, and let $\overline{S}=V-S$. Consider the following condition:
\begin{equation}\label{alliancecondition}
\delta_S(v)\ge \delta_{\overline{S}}(v)+k,
\end{equation}
which states that a vertex $v$ has at least $k$ more neighbors in $S$ than it has in $\overline{S}$.
A set $S\subseteq V$ that
satisfies Condition (\ref{alliancecondition}) for every vertex $v \in S$ is called a \emph{defensive} $k$-\emph{alliance}; for every vertex $v$ in the neighborhood of $S$ is called an \emph{offensive} $k$-\emph{alliance}. A subset of vertices $S\subseteq V$, is a \emph{powerful} $k$-\emph{alliance} if it is
both a defensive $k$-alliance and an offensive $(k +2)$-alliance. Moreover, a subset
$X\subset V$ is a defensive (an offensive or a powerful)
$k$-alliance free set if $X$ does not contain any defensive (offensive
or powerful, respectively) $k$-alliance.
In this article we study the relationships  between defensive (offensive, powerful) $k$-alliance free sets in Cartesian product graphs and defensive (offensive, powerful) $k$-alliance free sets in the factor graphs.
\end{abstract}



\section{Introduction}


The study of relationships between invariants of Cartesian product graphs and invariants of its
factor graphs appears frequently in researches about graph theory.
In this sense, there are important open problems which are being
investigated now. For instance, the Vizing's conjecture
\cite{bookdom1,bookdom2,vizing}, which is one of the most known
open problems in graph theory, states that the domination number of the
Cartesian product of two graphs is at least equal to the product of
the domination numbers of these two graphs. Some variations and
partial results about this conjecture have been developed in the last
years, like those in \cite{bresar1,bresar2,clark,sun}.

Apart from the domination number, there are several invariants which have been
studied in Cartesian product graphs. For instance, the
geodetic number
\cite{geodetic,bresar2,gera,pelayo2}, the metric dimension
\cite{pelayo1}, the partition dimension \cite{yerocartpartres}, the Menger number \cite{Menger}, the $k$-domination number \cite{kdomination}, the offensive $k$-alliance number  \cite{off2008},  the $k$-alliance partition number   \cite{c_nxc_m,defPartit} and the offensive $k$-alliance partition number \cite{PartOff}.

This article concerns the study of alliance free sets in Cartesian product graphs.
Since (defensive, offensive and powerful) alliances in graphs were first introduced by Kristiansen et al. \cite{alliancesOne}, several
authors have studied their mathematical properties \cite{off2008,c_nxc_m,partitionGrid,nosotros,tesisshafique,sd2009,kdaf1,kdaf2,kdaf3,GlobOff,PartOff,tesisisma,defPartit}  (the reader is referred to the Ph.D. Thesis \cite{tesisisma} for a more complete list of references). Applications of alliances can be found in the Ph. D. Thesis \cite{tesisshafique} where
the author studied problems of partitioning graphs into alliances and its
application to data clustering. On the other hand, defensive alliances represent
the mathematical model of web communities, by adopting the definition
of Web Community proposed by Flake et al. in \cite{flake}, ``a Web
Community is a set of web pages having more hyperlinks (in either direction)
to members of the set than to non-members". Other applications of alliances
were presented  in \cite{Bio} (where alliances were used in a quantitative analysis of
secondary RNA structure),  \cite{foodwebs} (where alliances were used in the study of  predator-prey
models on complex networks),  \cite{clynical} (where alliances were used in the study of  spatial models of cyclical
population interactions) and  \cite{monopoly} (where alliances were used as a model of monopoly).

In this work we continue the previous studies \cite{tesisshafique,sd2009,kdaf2,kdaf3,nosotros} on $k$-alliance free sets and $k$-alliance cover sets focusing our attention on the particular case of Cartesian product graphs.
We study the relationships  between defensive (offensive, powerful) $k$-alliance free sets in Cartesian product graphs and defensive (offensive, powerful) $k$-alliance free sets in the factor graphs.
The plan of the article is the following: In Section \ref{NT} we present the notation and terminology used and we recall the definitions of Cartesian product graph, defensive (offensive and powerful) $k$-alliance,  defensive (offensive and powerful) $k$-alliance free set and defensive (offensive and powerful) $k$-alliance cover set.
 Section \ref{Defensive} is devoted to the study of defensive $k$-alliances. More specifically, we give a sufficient condition for the existence of defensive $k$-alliance free sets in cartesian product graphs and  we study the relationships  between the maximum cardinality of a defensive $k$-alliance free set in Cartesian product graphs and several invariants of the factor graphs, including the order and the independence number.
  Analogously,   sections \ref{offensive} and \ref{Powerful}, respectively, are devoted to the study of offensive and powerful $k$-alliance free sets. In section \ref{concl} we present the conclusions.

\section{Notation and terminology} \label{NT} In this paper
$G=(V,E)$ denotes a simple graph of  order $n$,
minimum degree $\delta$ and maximum degree $\Delta$.  The independence number of $G$ is denoted  by $\alpha(G)$.  For a non-empty
set $S\subseteq V$  and a vertex $v\in V$,
$\delta_S(v)$ denotes the number of neighbors $v$ has in $S$  and $\delta(v)$  denotes  the degree of $v$.  The complement
of the set $S$ in $V$ is denoted by $\overline{S}$. The set of vertices of $\overline{S}$ which are adjacent to at least one vertex in $S$ is denoted by $\partial S$.

A non-empty set of vertices
$S\subseteq V$ is called a {\em defensive} (respectively, an \emph{offensive})
$k$-{\em alliance} in $G$ if for every $v\in S$ (respectively, $v\in
\partial S$), $\delta_S(v) \ge \delta_{\overline{S}}(v)+k,$ where
$k\in \{-\Delta,..., \Delta\}$ (respectively, $k\in \{2-\Delta,..., \Delta\}$). Also, a non-empty set of vertices $S\subseteq
V$ is called a {\em powerful} $k$-{\em alliance} in $G$ if it
is both, defensive $k$-alliance and offensive ($k+2$)-alliance, $k\in \{-\Delta,..., \Delta-2\}$. Notice that, since $V$ is an offensive $k$-alliance for every $k\in \{2-\Delta,..., \Delta\}$, $V$ is a powerful $k$-alliance if and only if it is a defensive $k$-alliance.

A set $X\subseteq V$ is ({\em defensive, offensive, powerful})
$k$-{\em alliance free},  ($k$-daf, $k$-oaf, $k$-paf), if
for all (defensive, offensive, powerful) $k$-alliance $S$,
$S\setminus X\neq\emptyset$, i.e., $X$ does not contain any
(defensive, offensive, powerful) $k$-alliance as a subset
\cite{kdaf1}.

Associated with the characteristic sets defined above we have the
following invariants:

\begin{itemize}

\item[] $\phi_k^d(G)$: maximum cardinality of a  $k$-daf set in $G$, $k\in \{-\Delta,..., \Delta\}$.

\item[]$\phi_k^o(G)$: maximum cardinality of a $k$-oaf  set in
$G$, $k\in \{2-\Delta,..., \Delta\}$.

\item[]  $\phi_k^{p}(G)$: maximum cardinality of a
$k$-paf  set in  $G$, $k\in \{-\Delta,..., \Delta-2\}$.
\end{itemize}

We now state the following fact on (defensive, offensive and powerful) $k$-alliance free sets that will be useful throughout this article.

\begin{remark}\label{monotoniafreeset}
If $X$ is a $k$-alliance free set and $k<k'$, then $X$ is a $k'$-alliance
free set.
\end{remark}

A set $Y \subseteq V$ is a ({\em defensive, offensive, powerful})
{\em $k$-alliance cover set} if
for all (defensive, offensive, powerful) $k$-alliance $S$, $S\cap
Y\neq\emptyset$, i.e., $Y$ contains at least one vertex from each
(defensive, offensive, powerful) $k$-alliance of $G$.

The following duality between $k$-alliance cover sets and $k$-alliance free sets allows us to study $k$-alliance cover sets from the results obtained on $k$-alliance free sets, so in this article we only consider the study of $k$-alliance free sets.

\begin{remark}\label{thDual}
$X$ is a $k$-alliance cover set if and
only if $\overline{X}$ is a $k$-alliance
free set.
\end{remark}

We recall that the Cartesian product of two graphs $G_1=(V_1,E_1)$ and
$G_2=(V_2,E_2)$ is the graph $G_1\times G_2=(V,E)$, such that
$V=\{(a,b)\;:\;a\in V_1,\;b\in V_2\}$ and two vertices $(a,b)\in V$
and $(c,d)\in V$ are adjacent in  $G_1\times G_2$ if and only if, either $a=c$ and $bd\in E_2$
or $b=d$ and $ac\in E_1$.

For a set $A\subseteq V_1\times V_2$ we denote by $P_{V_i}(A)$ the projection of $A$ over $V_i$, $i\in \{1,2\}$.

\section{Defensive $k$-alliance free sets in Cartesian product graphs} \label{Defensive}

To begin with the study we present the following straightforward result.
 \begin{remark} \label{remark1}
 Let $G_i$ be a graph of order $n_i$, minimum degree $\delta_i$ and maximum degree $\Delta_i$, $i\in \{1,2\}$. Then, for every $k \in\{ 1-\delta_1-\delta_2,...,\Delta_1+\Delta_2\}$,
$$\phi_{_k}^d(G_1\times G_2)\ge \alpha(G_1)\alpha(G_2)+\min\{n_1-\alpha(G_1),n_2-\alpha(G_2)\}.$$
\end{remark}

\begin{proof}
For every  graph $G$ of minimum degree $\delta$ and maximum degree $\Delta$,
any independent set in $G$ is a $k$-daf set for $k\in \{1-\delta, ...,\Delta\}$.
Hence, $\phi_{_k}^d(G_1\times G_2)\ge \alpha(G_1\times G_2)$, for every $k\in \{1-\delta_1-\delta_2, ...,\Delta_1+\Delta_2\}$, and by  the Vizing's inequality,
$\alpha(G_1\times G_2)\ge \alpha(G_1)\alpha(G_2)+\min\{n_1-\alpha(G_1),n_2-\alpha(G_2)\},$ we obtain the  result.
\end{proof}

Let $G_1$ be the star graph of order $t+1$ and let $G_2$ be the path graph of order $3$. In this case, $\phi_k^d(G_1\times G_2)=2t+1$ for $k\in \{-1,0\}$. Therefore, the above bound is tight. Even so, Corollary \ref{CoroTh1def(i)} (ii) improves the above bound for the cases where $\phi_{_{k_i}}^d(G_i)>\alpha(G_i)$, for some  $i\in \{1,2\}$.

\begin{theorem}\label{th1}
Let $G_i=(V_i,E_i)$ be a simple graph of maximum degree $\Delta_i$,  $i\in \{1,2\}$, and let $S\subseteq  V_1\times V_2 $.
Then the following assertions hold.
\begin{itemize}
\item[{\rm (i)}] If  $P_{V_i}(S)$ is a $k_i$-daf set in $G_i$, then $S$ is a $(k_i+\Delta_j)$-daf set in $G_1\times G_2$, where $j\in \{1,2\}$, $j\ne i.$
\item[{\rm (ii)}]
If for every  $i\in \{1,2\}$, $P_{V_i}(S)$ is a $k_i$-daf set in $G_i$,   then $S$ is a $(k_1+k_2-1)$-daf set in $G_1\times G_2$.

\end{itemize}
\end{theorem}

\begin{proof}
Let $A\subseteq S$ and we suppose $P_{V_1}(S)$ is a $k_1$-daf set in $G_1$. Since $P_{V_1}(A)\subseteq P_{V_1}(S)$,  there exists $a\in P_{V_1}(A)$ such that $\delta_{P_{V_1}(A)}(a)<\delta_{\overline{P_{V_1}(A)}}(a)+k_1$. If we take   $b\in V_2$ such that $(a,b)\in A$, then
$$\delta_A(a,b)\le \delta_{P_{V_1}(A)}(a)+\delta_{P_{V_2}(A)}(b)<\delta_{\overline{P_{V_1}(A)}}(a)+k_1+\delta(b)\le \delta_{\overline{P_{V_1}(A)}}(a)+k_1+\Delta_2.$$
Thus, $A$ is not a defensive $(k_1+\Delta_2)$-alliance in $G_1\times G_2$. Therefore, (i) follows.

In order to prove (ii), let $x\in X=P_{V_1}(A)$ such that $\delta_{X}(x)<\delta_{\overline{X}}(x)+k_1$.  Let $A_x\subseteq A$ be the set composed by the elements of $A$ whose first component is $x$. On the other hand, since  $P_{V_2}(S)$ is a $k_2$-daf set and $Y=P_{V_2}(A_x)\subseteq P_{V_2}(S)$, there exists $y \in Y$ such that $\delta_Y(y)<\delta_{\overline{Y}}(y)+k_2$. Notice that $(x,y)\in A$. Let $A_y\subseteq A$ be the set composed by the elements of $A$ whose second component is $y$.
Hence,
\begin{align*}\delta_A(x,y)&=\delta_{A_x}(x,y)+\delta_{A_y}(x,y)\\&\le\delta_Y(y)+\delta_X(x)\\
&<\delta_{\overline{Y}}(y)+\delta_{\overline{X}}(x)+k_1+k_2-1\\
&\le\delta_{\overline{A_x}}(x,y)-\delta(x)+\delta_{\overline{A_y}}(x,y)-\delta(y)+k_1+k_2-1\\
&\le\delta_{\overline{A_x}}(x,y)+\delta_{\overline{A_y}}(x,y)+k_1+k_2-1\\&= \delta_{\overline{A}}(x,y)+k_1+k_2-1.\end{align*}
Thus, $A$ is not a defensive $(k_1+k_2-1)$-alliance in $G_1\times G_2$ and, as a consequence, (ii) follows.
\end{proof}

\begin{corollary}\label{CoroTh1def(i)}Let $G_l$ be a graph of order $n_l$, maximum degree
$\Delta_l$ and minimum degree $\delta_l$, with $l\in \{1,2\}$. Then the following assertions hold.
\begin{itemize}
\item[{\rm (i)}]
  For every
$k\in \{\Delta_j-\Delta_i,..., \Delta_i+\Delta_j\}$ $(i,j\in \{1,2\}$, $i\ne j)$,
$$\phi_k^d(G_1\times G_2)\ge
n_j\phi_{k-\Delta_j}^d(G_i).$$

\item[{\rm (ii)}] For every $k_i \in\{ 1-\delta_i,...,\Delta_i\}$, $i\in \{1,2\}$,
$$\phi_{_{k_1+k_2-1}}^d(G_1\times G_2)\ge \phi_{_{k_1}}^d(G_1)\phi_{_{k_2}}^d(G_2)+\min\{n_1-\phi_{_{k_1}}^d(G_1),n_2-\phi_{_{k_2}}^d(G_2)\}.$$
\end{itemize}
\end{corollary}

\begin{proof}
By Theorem \ref{th1} (i)  we conclude that for every $k_i$-daf set $S_i$ in $G_i$, $i\in \{1,2\}$, the sets $S_1\times V_2$ and $V_1\times S_2$ are $(k_1+\Delta_2)$-daf  and  $(k_2+\Delta_1)$-daf, respectively,  in $G_1\times G_2$. Therefore, (i) follows.

In order to prove (ii), let $V_1=\{u_1,u_2,...,u_{n_1}\}$ and $V_2=\{v_1,v_2,...,v_{n_2}\}$. Moreover, let $S_i$ be a $k_i$-daf set of maximum cardinality  in $G_i$, $i\in \{1,2\}$. We suppose $S_1=\{u_1,...,u_r\}$ and $S_2=\{v_1,...,v_s\}$. By Theorem  \ref{th1} (ii) we deduce that $S_1\times S_2$ is a $(k_1+k_2-1)$-daf set in $G_1\times G_2$. Now
let $X=\{(u_{r+i},v_{s+i}), i=1,...,t\}$, where $t=\min\{n_1-r,n_2-s\}$ and let $S= X\cup (S_1\times S_2)$. Since, for every $x\in X$, $\delta_S(x)=0$ and $k_i > -\delta_i$, $i\in \{1,2\}$,
we obtain that $S$ is a $(k_1+k_2-1)$-daf set in $G_1\times G_2$.
Thus, $\phi_{_{k_1+k_2-1}}^d(G_1\times G_2)\ge |S|= \phi_{_{k_1}}^d(G_1)\phi_{_{k_2}}^d(G_2)+\min\{n_1-\phi_{_{k_1}}^d(G_1),n_2-\phi_{_{k_2}}^d(G_2)\}$.
\end{proof}

Now we state the following fact that will be useful for an easy understanding of several examples in this paper.

\begin{proposition}\label{remarktree}
Let $G$ be a graph of order $n$ and maximum degree $\Delta$. Then $\phi_k^d(G)=n$, for each of the following  cases:
\begin{itemize}
\item[{\rm (i)}] $G$ is a tree of maximum degree $\Delta\ge 2$ and $k\in \{2,...,\Delta\}$.
\item[{\rm (ii)}] $G$ is a planar graph of maximum degree $\Delta\ge 6$ and $k\in \{6,...,\Delta\}$.
\item[{\rm (iii)}] $G$ is a planar triangle-free graph of maximum degree $\Delta\ge 4$ and $k\in \{4,...,\Delta\}$.
\end{itemize}
\end{proposition}

\begin{proof}
Suppose $S$ is a defensive $k$-alliance in $G=(V,E)$. That is, for every $v\in S$, it follows
\begin{equation}\label{condiciondefensivagrado}
2\delta_S(v)\ge \delta(v)+k.
\end{equation}
     If  some vertex $v\in S$ satisfies  $\delta(v)<k$, then equation  (\ref{condiciondefensivagrado}) leads to $\delta_S(v) > \delta(v)$, a contradiction. Hence, for every $v\in S$ we have $\delta(v)\ge k$ and, as a consequence,  equation  (\ref{condiciondefensivagrado}) leads to $\delta_S(v)\ge k$. Now, let $m_s$ be the size of the subgraph induced by $S$. Then  we have
\begin{equation} \label{condicionmedidasubgrafo}2m_s=\displaystyle\sum_{v\in S}\delta_S(v)\ge k|S|.\end{equation} Case (i). Since $G$ is a tree, we obtain $2(|S|-1)\ge 2m_s\ge k|S|\ge 2|S|$, a contradiction.

For the cases (ii) and (iii) we have $|S|\ge 3$, due to that if $|S|\le 2$, then equation (\ref{condiciondefensivagrado}) leads to $2\ge \delta(v)+k$, a contradiction.
It is well-known that the size of a planar graph  of order $n'\ge 3$ is bounded above by $3(n'-2)$. Moreover, in the case of triangle-free graphs the bound is $2(n'-2)$. Therefore, in case (ii) we have $m_s\le 3(|S|-2)$ and, as a consequence, equation (\ref{condicionmedidasubgrafo}) leads to $6(|S|-2)\ge k|S|\ge 6|S|$, a contradiction. Analogously, in case (iii) we have $m_s\le 2(|S|-2)$ and, as a consequence, equation (\ref{condicionmedidasubgrafo}) leads to $4(|S|-2)\ge k|S|\ge 4|S|$, a contradiction.
\end{proof}

We emphasize that Corollary \ref{CoroTh1def(i)}  and Proposition \ref{remarktree} lead to infinite families of graphs whose Cartesian product satisfies $\phi_k^d(G_1\times G_2)=n_1n_2.$ For instance, if
$G_1$ is a tree of order $n_1$ and maximum degree $\Delta_1\ge 2$, $G_2$ is a graph of order $n_2$ and  maximum degree $\Delta_2$, and $k\in \{2+\Delta_2,...,\Delta_1+\Delta_2\}$, we have $\phi_k^d(G_1\times G_2)=\phi_{k-\Delta_2}^d(G_1)n_2=n_1n_2.$
In particular, if $G_2$ is a cycle graph, then $\phi_4^d(G_1\times G_2)=n_1n_2.$

Another example of equality in Corollary \ref{CoroTh1def(i)} (ii) is obtained, for instance, taking the
Cartesian product of the star graph $S_t$ of order $t+1$ and the path graph $P_r$ of order
$r$. In that case, for $G_1=S_t$ we have $\delta_1=1$,
$n_1=t+1$ and $\phi_0^d(G_1)=t$, and, for $G_2=P_r$ we have
$\delta_2=1$, $n_2=r$ and $\phi_1^d(G_2)=r-1$. Therefore,
$
\phi_0^d(G_1)\phi_1^d(G_2)+\min\{n_1-\phi_0^d(G_1),n_2-\phi_1^d(G_2)\}=t(r-1)+1.
$
On the other hand, it is not difficult to check that, if we take all
leaves belonging to the copies of $S_t$ corresponding to the first $r-1$
vertices of $G_2$ and we add the vertex of degree $t$ belonging to the
last copy of $S_t$, we obtain a maximum defensive $0$-alliance free set  of cardinality $t(r-1)+1$ in the
graph $G_1\times G_2$, that is,
$\phi_0^d(G_1\times G_2)=t(r-1)+1$. This example also shows
that this  bound is better than the  bound obtained in
Remark \ref{remark1}, which is $t\left\lceil\frac{r}{2}\right\rceil+1$. In this
particular case, both  bounds are equal if and only if $r=2$ or
$r=3$.

\begin{figure}[h]
\begin{center}\label{fig1}
\includegraphics[width=4.5cm]{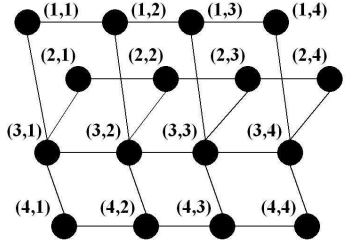}
\caption{This graph is the Cartesian product $S_3\times P_4$ where
$S=\{(1,1),(2,1),(4,1),(1,2),$ $(2,2),(4,2),(1,3),(2,3),(4,3),(3,4)\}$
is a maximum defensive $0$-alliance free set.}
\end{center}
\end{figure}

\begin{theorem}
Let $G_i=(V_i,E_i)$ be a graph and let $S_i\subseteq V_i$, $i\in
\{1,2\}$. If $S_1\times S_2$ is a $k$-daf set in
$G_1\times G_2$ and $S_2$ is a
defensive $k'$-alliance in $G_2$, then $S_1$ is a
$(k-k')$-daf set in $G_1$.
\end{theorem}

\begin{proof}
If $S\subseteq S_1$, then $S\times
S_2\subseteq S_1\times S_2$ is a $k$-daf set in
$G_1\times G_2$. So, there exists $(a,b)\in S\times S_2$
such that $\delta_{S\times S_2}(a,b)<\delta_{\overline{S\times
S_2}}(a,b)+k.$ Thus, we have
\begin{equation}\label{equat-1}
\delta_S(a)+\delta_{S_2}(b)=\delta_{S\times
S_2}(a,b)<\delta_{\overline{S\times
S_2}}(a,b)+k=\delta_{\overline{S}}(a)+\delta_{\overline{S_2}}(b)+k.
\end{equation}
As $S_2$ is a defensive $k'$-alliance in $G_2$, for every
$b\in S_2$ we have,
$\delta_{S_2}(b)\ge\delta_{\overline{S_2}}(b)+k'$. Hence, from
equation (\ref{equat-1}) we obtain
$\delta_S(a)<\delta_{\overline{S}}(a)+k-k'$. Therefore, $S$ is not a defensive $(k-k')$-alliance in $G_1$ and, as a consequence, $S_1$ is a $(k-k')$-daf set.
\end{proof}

Taking into account that $V_2$ is a defensive $\delta_2$-alliance in
$G_2$ we obtain the following result.

\begin{corollary}\label{otrocoro}
Let $G_i=(V_i,E_i)$ be a graph,  $i\in
\{1,2\}$. Let  $\delta_2$ be the minimum degree of $G_2$ and let $S_1\subseteq V_1$. If $S_1\times V_2$ is a $k$-daf set in
$G_1\times G_2$, then $S_1$ is a $(k-\delta_2)$-daf set in
$G_1$.
\end{corollary}

By Theorem \ref{th1} (i) and Corollary \ref{otrocoro} we obtain the following result.
\begin{proposition}
Let  $G_1$ be a graph of maximum degree $\Delta_1$ and let $G_2$ be a $\delta_2$-regular graph.   For every $k\in \{\delta_2-\Delta_1,...,\Delta_1+\delta_2\}$,  $S_1\times V_2$ is a $k$-daf
set in $G_1\times G_2$ if and only if $S_1$ is a
$(k-\delta_2)$-daf set in $G_1$.
\end{proposition}

\section{Offensive $k$-alliance free sets in Cartesian product  graphs} \label{offensive}

\begin{theorem} \label{th1of}
Let $G_i=(V_i,E_i)$ be a graph,
$i\in \{1,2\}$, and let $S\subset V_1\times V_2$. If  $P_{V_i}(S)$  is a $k$-oaf set in $G_i$, then $S$ is a
$(k-\delta_j)$-oaf set in $G_1\times G_2$, where $\delta_j$ denotes the minimum degree of $G_j$ and  $j\in \{1,2\}$, $i\ne j$.
\end{theorem}

\begin{proof}
If $P_{V_1}(S)$ is a $k$-oaf set in $G_1$ and $A\subseteq S$, then $P_{V_1}(A)\subseteq P_{V_1}(S)$ is a
$k$-oaf set in $G_1$. So, there exists $a\in \partial
P_{V_1}(A)$, such that
$\delta_{P_{V_1}(A)}(a)<\delta_{\overline{P_{V_1}(A)}}(a)+k$. Let $a'\in P_{V_1}(A)$ such that $a$ and $a'$ are adjacent, and let $Y_{a'}$ be the set of elements
of $A$ whose first component is $a'$. Thus, if $b\in P_{V_2}(Y_{a'})$, then   $(a,b)\in \partial A$, so we have
$$\delta_A(a,b)\le\delta_{P_{V_1}(A)}(a)<\delta_{\overline{P_{V_1}(A)}}(a)+k\le
\delta_{\overline{A}}(a,b)-\delta(b)+k\le\delta_{\overline{A}}(a,b)+k-\delta_2.$$
Therefore, $A$ is not an offensive ($k-\delta_2$)-alliance in
$G_1\times G_2$. The proof of the other case is completely analogous.
\end{proof}

From Theorem \ref{th1of} we conclude that for every $k_i$-oaf set $S_i$ in $G_i$, $i\in \{1,2\}$, the sets $S_1\times V_2$ and $V_1\times S_2$ are $(k_1-\delta_2)$-oaf  and  $(k_2-\delta_1)$-oaf,  respectively, in $G_1\times G_2$. Therefore, we obtain the following result.

\begin{corollary}\label{coronofensive}
Let $G_l$ be a graph of order $n_l$, maximum degree
$\Delta_l$ and minimum degree $\delta_l$,   $l\in \{1,2\}$.   Then, for every
$k\in \{2-\delta_j-\Delta_i,..., \Delta_i-\delta_j\}$, $\phi_k^o(G_1\times G_2)\ge
n_j\phi_{k+\delta_j}^o(G_i)$, where $i,j\in \{1,2\}$, $i\ne j$.
\end{corollary}

Example of equality in the above result is the following. If we take $G_1=C_4$, $G_2=P_3$ and $k_2=2$, then
$\phi_0^o(C_4\times P_3)=8=4
\phi_{2}^o(P_3).$
\begin{theorem}
Let $G_i=(V_i,E_i)$ be a graph of minimum degree $\delta_i$ and maximum degree $\Delta_i$. If $S_i$ is a $k_i$-oaf set in $G_i$, $i\in \{1,2\}$, then for  every $k\in \{k', ...,\Delta_1+\Delta_2\}$,  $(S_1\times V_2)\cup (V_1\times S_2)$ is a $k$-oaf set in $G_1\times G_2$,  where $ k'= \max\left\{k_1-\delta_2, k_2-\delta_1,\min\{k_2+\Delta_1,k_1+\Delta_2\}\right\}$.
\end{theorem}

\begin{proof}
Let $A\subseteq (S_1\times V_2)\cup (V_1\times S_2)$. By Theorem \ref{th1of} we deduce that, if $A\subseteq S_1\times V_2$, then $A$ is a $(k_1-\delta_2)$-oaf set in $G_1\times G_2$. Analogously, if $A\subseteq V_1\times S_2$, then $A$ is a $(k_2-\delta_1)$-oaf set in $G_1\times G_2$.

Now we suppose   $A\nsubseteq S_1\times V_2$ and  $A\nsubseteq V_1\times S_2$. Let $B=A\setminus (S_1\times V_2)$. For every $a\in P_{V_1}(B)$, the set $Y_a$, composed by the elements of $B$ whose first component is $a$, satisfies that $P_{V_2}(Y_a)$ is a $k_2$-oaf set in $G_2$. Then, there exists   $b\in \partial P_{V_2}(Y_a)$ such that  $\delta_{P_{V_2}(Y_a)}(b)<\delta_{\overline{P_{V_2}(Y_a)}}(b)+k_2$.  Also, notice that $(a,b)\in \partial A$. Thus,
$$\delta_A(a,b)\le \delta_{P_{V_2}(Y_a)}(b)+\delta(a)<\delta_{\overline{P_{V_2}(Y_a)}}(b)+k_2+\delta(a)\le \delta_{\overline{A}}(a,b)+k_2+\Delta_1.$$
We conclude that $A$ is not an offensive $(k_2+\Delta_1)$-alliance in $G_1\times G_2$. Analogously,  $A$ is not an offensive $(k_1+\Delta_2)$-alliance in $G_1\times G_2$. Therefore, the result follows.
\end{proof}

\begin{corollary}
 Let $G_i$ be a graph of order $n_i$, minimum degree $\delta_i$ and maximum degree $\Delta_i$, $i\in \{1,2\}$. Let $ k'=\max\left\{k_1-\delta_2, k_2-\delta_1,\min\{k_2+\Delta_1,k_1+\Delta_2\}\right\}$, where $k_i\in \{2-\Delta_i,...,\Delta_i\}$. Then, for every $k\in \{k',...,\Delta_1+\Delta_2\}$,
 $$\phi_k^o(G_1\times G_2)\ge
n_1\phi_{k_2}^o(G_2)+n_2\phi_{k_1}^o(G_1)-\phi_{k_1}^o(G_1)\phi_{k_2}^o(G_2).$$
\end{corollary}

For instance, if we take $G_1=C_3$, $G_2=P_3$,   $k_1=1$ and $k_2=2$, then
$\phi_3^o(C_3\times P_3)=7=3
\phi_{2}^o(P_3)+3\phi_{1}^o(C_3)-\phi_{1}^o(C_3)\phi_{2}^o(P_3).$

\begin{figure}[h]
\begin{center}\label{fig1of}
\includegraphics[width=4.5cm]{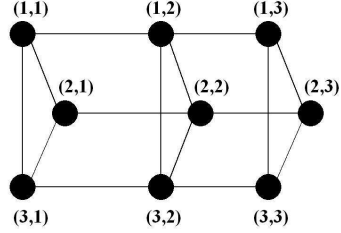}
\caption{The graph $G=(V,E)$ is the Cartesian product  of the cycle graph $C_3$ by the path graph $P_3$ where
$S=V \setminus \{(1,3),(2,3)\}$
is a maximum offensive $3$-alliance free set.}
\end{center}
\end{figure}

\section{Powerful $k$-alliance free sets in Cartesian product  graphs}  \label{Powerful}

Since for every graph $G$,  $\phi_{k}^p(G)\ge \max \{\phi_{k}^d(G),\phi_{k+2}^o(G)\}$, we have that lower bounds on $\phi_{k}^d(G)$ and $\phi_{k+2}^o(G)$ lead to lower bounds on  $\phi_{k}^p(G)$.
 So, by the results obtained in the above sections on $\phi_k^d(G_1\times G_2)$ and
 $\phi_{k+2}^o(G_1\times G_2)$ we deduce lower bounds on $\phi_k^p(G_1\times G_2)$.
\begin{figure}[h]\label{fignopowerful}
\begin{center}
\includegraphics[width=4.5cm]{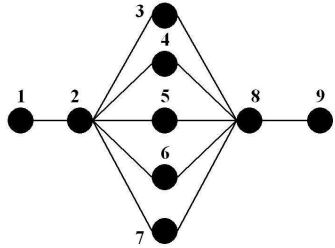}
\vspace{-0,2cm}
\caption{A graph $G=(V,E)$ where $V$ is a powerful  $2$-alliance free set, although $\{2,3,4,5,6,8\}$ is a defensive $2$-alliance and $\{3,4,5,6,7\}$ is an offensive  $4$-alliance.}
\end{center}
\end{figure}

We emphasize that there are graphs where $\phi_{k}^p(G)> \max \{\phi_{k}^d(G),\phi_{k+2}^o(G)\}$. For instance,  the graph of above figure satisfies $\phi_{2}^p(G)=9$ while $\phi_{2}^d(G)=8$ and $\phi_{4}^o(G)=7$.

\begin{theorem}\label{th1p}
Let $G_i=(V_i,E_i)$ be a simple graph of maximum degree $\Delta_i$ and minimum degree $\delta_i$,  $i\in \{1,2\}$, and let $S\subseteq  V_1\times V_2 $.
Then the following assertions hold.
\begin{itemize}
\item[{\rm (i)}] If  $P_{V_i}(S)$ is a $k_i$-paf set in $G_i$, then, for every $k\in \{k_i+\Delta_j,...,\Delta_i+\Delta_j-2\}$,  $S$ is a $k$-paf set in $G_1\times G_2$, where $j\in \{1,2\}$, $j\ne i.$
\item[{\rm (ii)}]
If for every  $i\in \{1,2\}$, $P_{V_i}(S)$ is a $k_i$-paf set in $G_i$,   then, for every $k\in \{k', ..., \Delta_1+\Delta_2-2\}$, $S$ is a $k$-paf set in $G_1\times G_2$, where $k'=\max\{k_1+k_2-1,\min\{ k_2-\delta_1,k_1-\delta_2\}\}$.
\end{itemize}
\end{theorem}

\begin{proof}
Let $A\subseteq S$. We suppose $P_{V_i}(S)$ is a $k_i$-paf set in $G_i$ for some $i\in \{1,2\}$. Since $P_{V_i}(A)\subseteq P_{V_i}(S)$,  it follows that  $P_{V_i}(A)$ is not a powerful $k_i$-alliance in $G_i$.  If $P_{V_i}(A)$ is not a defensive $k_i$-alliance, by analogy to the proof of Theorem \ref{th1} (i), we obtain that $A$ is not a defensive $(k_i+\Delta_j)$-alliance in $G_1\times G_2$, $j\ne i$. If $P_{V_i}(A)$ is not an offensive $(k_i+2)$-alliance in $G_i$, then by analogy to the proof of Theorem \ref{th1of}, we obtain that $A$ is not an offensive $(k_i-\delta_j+2)$-alliance in $G_1\times G_2$, $j\ne i$. Since, $k_i+\Delta_j>k_i-\delta_j $, we obtain that $A$ is not a powerful $(k_i+\Delta_j)$-alliance in $G_1\times G_2$. Therefore, (i)  follows.

If for every  $l\in \{1,2\}$, $P_{V_l}(S)$ is a $k_l$-paf set in $G_l$,  then  $P_{V_l}(A)$ is not a powerful $k_l$-alliance in $G_l$.
Hence, we differentiate two cases.

Case (1): For some  $l\in \{1,2\}$, $P_{V_l}(A)$  is not a defensive $k_l$-alliance. We suppose $P_{V_1}(A)$  is not a defensive $k_1$-alliance. Hence,
there exists $x\in P_{V_1}(A)$ such that $\delta_{P_{V_1}(A)}(x)<\delta_{\overline{P_{V_1}(A)}}(x)+k_1$. Let $A_x\subseteq A$ be the set composed by the elements of $A$ whose first component is $x$. If $P_{V_2}(A_x)\subset P_{V_2}(S)$ is not a defensive $k_2$-alliance, then
 by analogy to the proof of Theorem \ref{th1} (ii), we obtain that
 $A$ is not a defensive $(k_1+k_2-1)$-alliance in $G_1\times G_2$. On the other hand, if $P_{V_2}(A_x) $ is a defensive $k_2$-alliance, then it is not an offensive $(k_2+2)$-alliance. Thus, there exists $y\in \partial P_{V_2}(A_x)$ such that $\delta_{P_{V_2}(A_x)}(y)<\delta_{\overline{P_{V_2}(A_x)}}(y)+(k_2+2)$.
 We note that $(x,y)\in \partial A$. Hence,
 \begin{align*}
  \delta_A(x,y)&\le \delta_{P_{V_1}(A)}(x)+\delta_{P_{V_2}(A_x)}(y)\\
  &< \delta_{\overline{P_{V_1}(A)}}(x)+\delta_{\overline{P_{V_2}(A_x)}}(y)+k_1+k_2+1\\
  &\le \delta_{\overline{A}}(x,y)+k_1+k_2+1.
 \end{align*}
 As a consequence,  $A$ is not an offensive $(k_1+k_2+1)$-alliance in $G_1\times G_2$. Thus, in this case, $A$ is not a powerful $(k_1+k_2-1)$-alliance in $G_1\times G_2$.

 Case (2): For every  $i\in \{1,2\}$, $P_{V_i}(A)$  is not an offensive $(k_i+2)$-alliance in $G_i$. In this case, as we have shown in the proof of (i),  $A$ is not an offensive $(k_i-\delta_j+2)$-alliance in $G_1\times G_2$, $j\in \{1,2\}$, $j\ne i$.

 As a consequence, for $k= \max\{ k_1+k_2-1,k_1-\delta_2,k_2-\delta_1\}$, $A$ is not a powerful $k$-alliance in $G_1\times G_2$. Hence, $S$ is a  $k$-paf set in $G_1\times G_2$.  Therefore, (ii)  follows.
\end{proof}

\begin{corollary}\label{coroproductpowerful}Let $G_l$ be a graph of order $n_l$, maximum degree
$\Delta_l$ and minimum degree $\delta_l$,
 $l\in \{1,2\}$.  Let $k_l \in\{ 1-\delta_l,...,\Delta_l-2\}$. Then the following assertions hold.

\begin{itemize}
\item[{\rm (i)}]
For every
$k\in \{\Delta_j-\Delta_i,..., \Delta_i+\Delta_j-2\}$, $(i,j\in \{1,2\}$, $i\ne j)$
$$\phi_k^p(G_1\times G_2)\ge
n_j\phi_{k-\Delta_j}^p(G_i).$$
\item[{\rm (ii)}]  For every $k\in \{k_1+k_2-1,...,\Delta_1+\Delta_2-2\}$,
$$\phi_{_{k}}^p(G_1\times G_2)\ge \phi_{_{k_1}}^p(G_1)\phi_{_{k_2}}^p(G_2)+\min\{n_1-\phi_{_{k_1}}^p(G_1),n_2-\phi_{_{k_2}}^p(G_2)\}.$$
\end{itemize}
\end{corollary}

\begin{proof}
By Theorem \ref{th1p} (i)  we conclude that for every $k_i$-paf set $S_i$ in $G_i$, $i\in \{1,2\}$, the sets $S_1\times V_2$ and $V_1\times S_2$ are, respectively,  $(k_1+\Delta_2)$-paf  and  $(k_2+\Delta_1)$-paf in $G_1\times G_2$. Therefore, (i) follows.

In order to prove (ii), let
 $V_1=\{u_1,u_2,...,u_{n_1}\}$ and $V_2=\{v_1,v_2,...,v_{n_2}\}$. Let $S_i$ be a $k_i$-paf set of maximum cardinality  in $G_i$, $i\in \{1,2\}$. We suppose $S_1=\{u_1,...,u_r\}$ and $S_2=\{v_1,...,v_s\}$. By Theorem  \ref{th1p} (ii) we deduce that, for $k\ge k_1+k_2-1$, $S_1\times S_2$ is a $k$-paf set in $G_1\times G_2$. Now
let $X=\{(u_{r+i},v_{s+i}), i=1,...,t\}$, where $t=\min\{n_1-r,n_2-s\}$ and let $S= X\cup (S_1\times S_2)$. Since,
for every $x\in X$, $\delta_S(x)=0$ and $k_i> -\delta_i$, $i\in \{1,2\}$,
we obtain that for every $A\subseteq S$, such that $A\cap X\ne \emptyset$, $A$ is not a defensive $(k_1+k_2-1)$-alliance  in $G_1\times G_2$. Hence, $S$ is a $k$-paf set for $k\ge k_1+k_2-1$. As a consequence,
 $\phi_{k}^p(G_1\times G_2)\ge |S|= \phi_{_{k_1}}^p(G_1)\phi_{_{k_2}}^p(G_2)+\min\{n_1-\phi_{_{k_1}}^p(G_1),n_2-\phi_{_{k_2}}^p(G_2)\}$.
\end{proof}

If  $G_1=C_{n_{_1}}$ is the cycle graph of order $n_1$  and $G_2$ is the graph in Figure 3, then,  by Corollary \ref{coroproductpowerful} (i),  we deduce  $\phi_4^p(G_1\times G_2)=n_1n_2$, that is, $\phi_4^p(G_1\times G_2)\ge n_1\phi_2^p(G_2)=n_1n_2$. Moreover, if  $G_1=T_{n_1}$ is a tree of order $n_1$ and maximum degree $\Delta_1\ge 4$ and $G_2$ is the graph in Figure 3, then  $\phi_2^p(G_1)=n_1$ and $\phi_2^p(G_2)=n_2=9$. Therefore,  by Corollary \ref{coroproductpowerful} (ii)  we deduce  $\phi_3^p(G_1\times G_2)=9n_1$.

\section{Conclusions} \label{concl}
This article is a contribution to the study of alliances in graphs. Particularly, we have dealt with defensive (offensive, powerful) $k$-alliance free sets in Cartesian product graphs. We have  shown several relationships between defensive (offensive, powerful) $k$-alliance free sets in Cartesian product graphs and  defensive (offensive, powerful) $k$-alliance free sets in the factor graphs. Our principal contributions are summarized below.

\noindent
Let $G_i=(V_i,E_i)$ be a graph of maximum degree $\Delta_i$ and minimum degree $\delta_i$,  $i\in \{1,2\}$:
\begin{itemize}

\item We have shown that if the projection of a set $S\subset V_1\times V_2$ over $V_i$ is a  defensive (offensive, powerful) $k_i$-alliance free set in $G_i$, then $S$ is a defensive (offensive, powerful) $k$-alliance free set in $G_1\times G_2$, where the values of $k$ depend on the values of $k_i$, $\delta_j$ and $\Delta_j$, with $j\in \{1,2\}$.

\item We have shown the relationships between the maximum cardinality of a defensive (offensive, powerful) $k_i$-alliance free set in $G_i$ and the maximum cardinality of a defensive (offensive, powerful) $k$-alliance free set in  $G_1\times G_2$, where the values of $k$ depend on the values of $K_i$, $\delta_j$ and $\Delta_j$, with $j\in \{1,2\}$.
\end{itemize}



\end{document}